\documentclass{amsart}

\newtheorem{theorem}{Theorem}[section] 
\newtheorem{lemma}[theorem]{Lemma}
\newtheorem{corollary}[theorem]{Corollary}

\theoremstyle{definition}

\newtheorem{example}[theorem]{Example}

\theoremstyle{remark}

\numberwithin{equation}{section}

\begin{document}

\title{Tempered Representations and Nilpotent Orbits}

\author{Benjamin Harris}
\address{Department of Mathematics, Massachusetts Institute of Technology, Cambridge, Massachusetts 02139}
\email{blharris@math.mit.edu}



\date{August 21, 2011.}



\begin{abstract}
Given a nilpotent orbit $\mathcal{O}$ of a real, reductive algebraic group, a necessary condition is given for the existence of a tempered representation $\pi$ such that $\mathcal{O}$ occurs in the wave front cycle of $\pi$. The coefficients of the wave front cycle of a tempered representation are expressed in terms of volumes of precompact submanifolds of an affine space. 
\end{abstract}

\maketitle

\section{Introduction}

\indent Let $G$ be a real, reductive algebraic group with Lie algebra $\mathfrak{g}$, and let $\pi$ be an irreducible, admissible representation of $G$. The wave front cycle of $\pi$, denoted $\text{WF}(\pi)$, is an integral linear combination of nilpotent coadjoint orbits defined in \cite{BV}, \cite{SV}. 

\begin{theorem} Suppose $\mathcal{O}$ is an orbit contained in $\text{WF}(\pi)$ for a tempered representation $\pi$, let $\nu\in \mathcal{O}$, and let $L$ be a Levi factor of $Z_G(\nu)$. Then $L/Z(G)$ is compact.
\end{theorem}
\indent This theorem was conjectured by David Vogan. It makes precise the widely held intuition that orbits occuring in the wave front cycles of tempered representations are necessarily large.\\
\\
\indent In \cite{BC} Bala and Carter attempted a uniform understanding of nilpotent orbits in complex semisimple Lie algebras. They observed that for every nilpotent orbit $\mathcal{O}$, there is an unique (up to conjugacy) maximal complex Levi subalgebra $\mathfrak{l}$ meeting $\mathcal{O}$. Moreover, $\mathcal{O}\cap \mathfrak{l}$ is a single orbit for the adjoint group of $\mathfrak{l}$. Bala and Carter defined a nilpotent orbit to be \emph{distinguished} if it does not meet a proper Levi subalgebra, and they observed that to parametrize nilpotent orbits it is enough to parametrize distinguished nilpotent orbits in every Levi subalgebra.\\  
\indent Motivated by the work of Bala and Carter, we define a nilpotent orbit for a real reductive algebraic group to be \emph{real distinguished} if it does not meet a proper Levi subalgebra. By Levi subalgebra, we mean the Levi factor of a real parabolic subalgebra. Using the notion of a \emph{real distinguished} nilpotent orbit, one could attempt a Bala-Carter style classification of real nilpotent orbits. No\"{e}l gave a different generalization of Bala-Carter theory for real nilpotent orbits using the notion of a \emph{noticed} nilpotent orbit which is different from the notion of a \emph{real distinguished} nilpotent orbit \cite{No}.\\ 
\\
\indent If $\mathcal{O}_{\nu}$ is a nilpotent orbit, then one checks that $Z_G(\nu)$ is compact modulo $Z(G)$ iff $\mathcal{O}_{\nu}$ is a real distinguished nilpotent orbit by utilizing the fact that every real Levi subalgebra can be written as the centralizer of a hyperbolic element. Thus, we have an alternate way of stating Theorem 1.1. 


\begin{corollary} If $\mathcal{O}\subset \mathfrak{g}^*$ is a nilpotent coadjoint orbit occurring in the wave front cycle of a tempered representation, then $\mathcal{O}$ does not meet $\mathfrak{m}^*$ for any proper Levi subalgebra $\mathfrak{m}\subset \mathfrak{g}$.
\end{corollary}


Note that this corollary relates a set of nilpotent orbits that arises naturally in representation theory with a set of nilpotent orbits that arises naturally in structure theory.\\
\indent For $p$-adic groups, similar results were obtained much earlier. In \cite{M}, Moeglin shows that the $p$-adic analogue for supercuspidal representations follows from results of \cite{MW}. Similarly, Moeglin proves the result for tempered representations of classical $p$-adic groups in \cite{M}.\\
\\
\indent The purpose of this paper is to show that there is a simple proof of Theorem 1.1, which is in the spirit of the Kostant-Kirillov orbit method.\\
\\
\indent In \cite{R4} and \cite{R2}, Rossmann associated to each irreducible, tempered representation $\pi$ a finite union of regular coadjoint orbits we call $\mathcal{O}_{\pi}$. If $\pi$ has regular infinitesimal character, then $\mathcal{O}_{\pi}$ is a single coadjoint orbit. For each nilpotent coadjoint orbit, $\mathcal{O}\subset \mathfrak{g}^*$, fix $X\in \mathcal{O}$. Identify $\mathfrak{g}\cong \mathfrak{g}^*$ via a $G$-equivariant isomorphism, and let $\{X,H,Y\}$ be an $\mathfrak{sl}_2$-triple containing $X$. Put $S_X=X+Z_{\mathfrak{g}}(Y)$.

\begin{theorem} There exists a canonical measure on $S_X\cap \mathcal{O}_{\pi}$ such that
$$\text{WF}(\pi)=\sum_{\mathcal{O}_{\pi}\cap S_X\ \text{precompact}} \emph{vol}(\mathcal{O}_{\pi}\cap S_X)\mathcal{O}_X.$$
\noindent The sum is over nilpotent coadjoint orbits $\mathcal{O}_X$ such that $\mathcal{O}_{\pi}\cap S_X$ is precompact.
\end{theorem}

\indent If $\pi$ has regular infinitesimal character, then `precompact' may be replaced by `compact' in the above theorem. In the case where $G$ is compact, the wave front cycle of $\pi$ is $\dim(\pi)\cdot 0$ where $0$ denotes the zero orbit. In this case, our formula reduces to the well known observation of Kirillov that the symplectic volume of the coadjoint orbit associated to $\pi$ is the dimension of $\pi$ (see for instance page 173 of \cite{K}).\\
\\
\indent The reader will see in section four that the compactness in Theorem 1.1 is a consequence of the compactness of the manifold $\mathcal{O}_{\pi}\cap S_X$ in the case of regular infinitesimal character.\\
\\
\indent All of these results should be true for an arbitrary reductive Lie group. We only state them in the case of real, reductive algebraic groups because we use results of \cite{R4}, \cite{R2}, and \cite{R3}, which are only stated for real, reductive algebraic groups. However, the author believes that these results should be true for an arbitrary reductive Lie group.

\section{Rao's Lemma and three Corollaries of Barbasch and Vogan}

\indent Let $\mathcal{O}\subset \mathfrak{g}^*$ be a coadjoint orbit. The Kostant-Kirillov symplectic form $\omega$ is defined on $\mathcal{O}$ by the formula
$$\omega_{\lambda}(\text{ad}_X^*\lambda,\text{ad}_Y^*\lambda)=\lambda([X,Y]).$$
The top dimensional form
$$\frac{\omega^m}{m! (2\pi)^m}$$
on $\mathcal{O}$ gives rise to the canonical measure on $\mathcal{O}$. Here $m=\frac{\dim \mathcal{O}}{2}$. We will often abuse notation and write $\mathcal{O}$ for the orbit as well as the canonical measure on the orbit.\\
\indent In this section, we recall an unpublished lemma of Rao and three corollaries of Barbasch and Vogan. All of this material can be found on pages 46, 47, and 48 of \cite{BV}. However, unlike the previous treatment, we need to carefully keep track of certain constants for our applications. Thus, we provide updated statements, and for the convenience of the reader we provide sketches of updated proofs.\\
\indent Identify $\mathfrak{g}\cong \mathfrak{g}^*$ via a $G$-equivariant isomorphism. Let $\mathcal{O}_X$ be a nilpotent orbit in $\mathfrak{g}^*\cong \mathfrak{g}$, and let $\{X,H,Y\}$ be an $\mathfrak{sl}_2$-triple with nilpositive element $X$. Put $S_X=X+Z_{\mathfrak{g}}(Y)$. The map 
$$\phi: G\times S_X\rightarrow \mathfrak{g}^*$$ 
given by $\phi:(g,\xi)\mapsto g\cdot \xi$ is a submersion. In particular, every orbit $\mathcal{O}_{\nu}\subset G\cdot S_X$ is transverse to $S_X$, and $G\cdot S_X\subset \mathfrak{g}^*$ is open.\\
\indent Fix a Haar measure on $G$. This choice determines a Lebesgue measure on $\mathfrak{g}\cong \mathfrak{g}^*$. If $\xi\in S_X$, then we have a direct sum decomposition 
$$\mathfrak{g}=[\mathfrak{g},X]\oplus Z_{\mathfrak{g}}(Y)\cong T_X\mathcal{O}_X\oplus T_{\xi}S_X.$$ 
We then obtain a Lebesgue measure on $S_X$ as the `quotient' of the Lebesgue measure on $\mathfrak{g}$ and the canonical measure on $\mathcal{O}_X\subset \mathfrak{g}^*$. Further, given $\nu\in \mathfrak{g}^*$, denote by $\mathcal{F}_{\nu}$ the fiber over $\nu$ under the map $\phi$. If $g\cdot \xi=\nu$, then we have an exact sequence
$$0\rightarrow T_{\nu}^*(G\cdot S_X) \rightarrow T_{(g,\xi)}^*(G\times S_X)\rightarrow T_{(g,\xi)}^*\mathcal{F}_{\nu}\rightarrow 0.$$
This exact sequence together with the above remarks and our choice of Haar measure on $G$ determine a smooth measure on $\mathcal{F}_{\nu}$. Moreover, integration against these measures on the fibers of $\phi$ yields a continuous surjective map
$$\phi_*:\ C_c^{\infty}(G\times S_X)\longrightarrow C_c^{\infty}(G\cdot S_X).$$
\noindent Dualizing, we get an injective pullback map on distributions
$$\phi^*:\ D(G\cdot S_X)\rightarrow D(G\times S_X).$$
\noindent Now, we are ready to state Rao's lemma.

\begin{lemma} [Rao] If $\nu\in S_X$, then there exists a smooth measure $m_{\nu,X}$ on $\mathcal{O}_{\nu}\cap S_X$ such that 
$$\phi^*(\mathcal{O}_{\nu})=m_G\otimes m_{\nu,X}.$$
Here $m_G$ denotes the fixed choice of Haar measure on $G$. Although $\phi^*$ depends on this choice of Haar measure, $m_{\nu,X}$ does not.
\end{lemma}

\indent One can write down $m_{\nu,X}$ by giving a top dimensional form on $\mathcal{O}_{\nu}\cap S_X$, well-defined up to sign. Essentially, we just divide the canonical measure on $\mathcal{O}_{\nu}$ by the canonical measure on $\mathcal{O}_X$. More precisely, the composition of the inclusion $[\mathfrak{g},\nu]\hookrightarrow \mathfrak{g}$ and the projection defined by the decomposition $\mathfrak{g}=[\mathfrak{g},X]\oplus Z_{\mathfrak{g}}(Y)$ yields a map
$$T_{\nu}\mathcal{O}_{\nu}\cong [\mathfrak{g},\nu]\rightarrow [\mathfrak{g},X]\cong T_X\mathcal{O}_X.$$
This map is a surjection because $\mathcal{O}_{\nu}$ is transverse to $S_X$. It pulls back to an exact sequence
$$0\rightarrow T_X^*\mathcal{O}_X\rightarrow T_{\nu}^*\mathcal{O}_{\nu}\rightarrow T^*_{\nu}(\mathcal{O}_{\nu}\cap S_X)\rightarrow 0.$$
\noindent The canonical measures on $\mathcal{O}_{\nu}$ and $\mathcal{O}_X$ determine top dimensional alternating tensors up to sign on $T_{\nu}^*\mathcal{O}_{\nu}$ and $T_X^*\mathcal{O}_X$. Hence, our exact sequence gives a top dimensional, alternating tensor on $T^*_{\nu}(\mathcal{O}_{\nu}\cap S_X)$, well-defined up to sign.\\
\\ 
\indent We will need three corollaries of Barbasch and Vogan. Let $\mathcal{N}\subset \mathfrak{g}^*$ be the nilpotent cone. If $\nu\in \mathfrak{g}^*$, define
$$\mathcal{N}_{\nu}=\mathcal{N}\cap \overline{\cup_{t>0} \mathcal{O}_{t\nu}}.$$

\begin{corollary} [Barbasch and Vogan] We have four statements.\\
(a) If $\mathcal{O}_X$ is a nilpotent orbit, then $\mathcal{O}_X\subset \mathcal{N}_{\nu}$ if, and only if $\mathcal{O}_{\nu}\cap S_X\neq \emptyset$.\\
(b) An orbit $\mathcal{O}_X\subset \mathcal{N}_{\nu}$ is open if, and only if $\mathcal{O}_{\nu}\cap S_X$ is precompact.\\
(c) If $\nu$ is semisimple, then $\mathcal{O}_X\subset \mathcal{N}_{\nu}$ is open if, and only if $\mathcal{O}_{\nu}\cap S_X$ is compact.\\
(d) Further, $\mathcal{O}_X\cap S_X=\{X\}$ for any nilpotent orbit $\mathcal{O}_X$.
\end{corollary}

\indent We sketch a proof. Note $G\cdot S_X\subset \mathfrak{g}^*$ is an open subset containing $\mathcal{O}_X$; thus, $\mathcal{O}_X\subset \mathcal{N}_{\nu}$ iff $\mathcal{O}_{t\nu}\cap S_X\neq \emptyset$ for sufficiently small $t>0$. However, if $\gamma_t=\text{exp}(-\frac{1}{2}(\log(t))H)$, then  $$\mathcal{O}_{t\nu}\cap S_X=X+t\gamma_t(\mathcal{O}_{\nu}\cap S_X-X).$$ In particular, $\mathcal{O}_{\nu}\cap S_X \neq \emptyset$ iff $\mathcal{O}_{t\nu}\cap S_X\neq \emptyset$ for any $t>0$. This verifies part (a).\\
\indent For parts (b) and (c), one shows that $\mathcal{O}_{\nu}\cap S_X$ bounded implies that $S_X\cap \mathcal{N}_{\nu}=\{X\}$ and $\mathcal{O}_{\nu}\cap S_X$ unbounded implies that $S_X\cap \mathcal{N}_{\nu}$ is unbounded. This follows from a straightforward calculation utilizing the $\text{ad}_H$-decomposition of $Z_{\mathfrak{g}}(Y)$ into eigenspaces with non-positive eigenvalues and the above relationship between $\mathcal{O}_{\nu}\cap S_X$, $\mathcal{O}_{t\nu}\cap S_X$, and $\gamma_t$. Using that $(G\cdot S_X)\cap \mathcal{N}$ is the union of nilpotent orbits containing $\mathcal{O}_X$ in their closures, parts (b) and (c) follow.\\
\indent If we let $\nu=X$ in the last paragraph, we arrive at part (d). 

\begin{corollary} [Barbasch and Vogan] Let $n=\frac{1}{2}(\dim\mathcal{O}_{\nu}-\dim \mathcal{N}_{\nu})$. Then 
$$\lim_{t\rightarrow 0^+} t^{-n}\mathcal{O}_{t\nu}=\sum_{\substack{\mathcal{O}_X\subset \mathcal{N}_{\nu}\\ \dim\mathcal{O}_X=\dim\mathcal{N}_{\nu}}} \emph{vol}(\mathcal{O}_{\nu}\cap S_X)\mathcal{O}_X.$$
The volumes are computed with respect to the measures defined in Lemma 2.1. Moreover, the Fourier transform of the right hand side is the first non-zero term in the asymptotic expansion of the generalized function $\widehat{\mathcal{O}_{\nu}}$.
\end{corollary}

\indent Again, we sketch a proof. 
Fix $X$, a nilpotent element with $\mathcal{O}_X\subset \mathcal{N}_{\nu}$ open, and let $m=\frac{1}{2}\left(\dim \mathcal{O}_{\nu}-\dim \mathcal{O}_X\right)$. We first show 
$$\lim_{t\rightarrow 0^+} t^{-m}\mathcal{O}_{t\nu}= \text{vol}(\mathcal{O}_{\nu}\cap S_X)\mathcal{O}_X$$
on the open set $G\cdot S_X$. By Rao's lemma, $\phi^*(\mathcal{O}_{\nu})=m_G\otimes m_{\nu,X}$. Thus, it is enough to show $$\lim_{t\rightarrow 0^+} t^{-m} m_G\otimes m_{t\nu,X}=\text{vol}(\mathcal{O}_{\nu}\cap S_X) m_G\otimes \delta_X$$
if $\mathcal{O}_X\subset \mathcal{N}_{\nu}$ is an open orbit.
Note that the support of the measure $m_{t\nu,X}$ is the precompact set $\mathcal{O}_{t\nu}\cap S_X$, and the precompact sets $\mathcal{O}_{t\nu}\cap S_X$ converge uniformly to $\delta_X$ by the above relationship between $\mathcal{O}_{\nu}\cap S_X$, $\mathcal{O}_{t\nu}\cap S_X$, and $\gamma_t$. Thus, it is enough to show $$t^{-m} \text{vol}(\mathcal{O}_{t\nu}\cap S_X)=\text{vol}(\mathcal{O}_{\nu}\cap S_X).$$
This follows from a straightforward computation utilizing the above definition of the measures $m_{t\nu,X}$.\\
\indent By Theorem 3.2 of \cite{BV}, the distribution $\widehat{\mathcal{O}_{\nu}}$ has an asymptotic expansion at the origin
$$t^r\widehat{\mathcal{O}_{t\nu}}\sim t^lD_l+\cdots$$
where $D_l$ is the leading term and $r$ is the number of positive roots of $G$. If we show $n=l-r$, then our limit will exist everywhere. If $n>l-r$, then the limit $\lim_{t\rightarrow 0^+} t^n\widehat{\mathcal{O}_{t\nu}}$ must be zero everywhere. However, we have seen that the Fourier transform of this limit is nonzero on $\mathcal{O}_X$ whenever $\mathcal{O}_X\subset \mathcal{N}_{\nu}$ is of maximal dimension.\\
\indent The limit $\lim_{t\rightarrow 0^+} t^{-l+r}\mathcal{O}_{t\nu}$ must exist and be nonzero. However, if $n<l-r$, then the homogeneity degree of this invariant distribution and Corollary 3.9 of \cite{BV} imply that such a distribution would have to be supported on orbits of dimension greater than $\dim \mathcal{N}_{\nu}$. But, this is impossible since the limit $\lim_{t\rightarrow 0^+} t^{-l+r}\mathcal{O}_{t\nu}$, if it exists, is clearly supported in $\mathcal{N}_{\nu}$. Hence, $n=l-r$ and the limit $\lim_{t\rightarrow 0^+}t^{-n}\mathcal{O}_{t\nu}$ exists.\\
\indent Now, let $k=\dim \mathcal{N}_{\nu}$, and let $\mathcal{N}_k$ be the union of nilpotent orbits of dimension at least $k$. We have shown that our desired limit formula holds on $\mathcal{N}_k$. However, in theory the limit could differ from $\displaystyle \sum_{\substack{\mathcal{O}_X\subset \mathcal{N}_{\nu}\\ \dim\mathcal{O}_X=\dim\mathcal{N}_{\nu}}} \text{vol}(\mathcal{O}_{\nu}\cap S_X)\mathcal{O}_X$ by a distribution $u$ supported on orbits of dimension less than $\mathcal{N}_{\nu}$. However, we deduce $u=0$ from Corollary 3.9 of \cite{BV} after checking the homogeneity degree of the terms in our limit formula.

\begin{corollary} [Barbasch and Vogan] Suppose $\mathfrak{h}\subset \mathfrak{g}$ is a Cartan, and let $C\subset (\mathfrak{h}^*)'$ be a connected component of the regular set. If $\nu,\lambda\in C$, then $\mathcal{N}_{\nu}=\mathcal{N}_{\lambda}$.
\end{corollary}

Suppose $\mathcal{O}_X\subset \mathcal{N}_{\nu}$ is an open orbit, let $n=\frac{1}{2}(\dim \mathcal{O}_{\nu}-\dim \mathcal{O}_X)$, and observe $\lim_{t\rightarrow 0^+} t^{-n}\mathcal{O}_{t\nu}=\lim_{t\rightarrow 0^+} \frac{1}{n!}\partial(\nu)^n|_{t\nu}\mathcal{O}_{t\nu}$ on $G\cdot S_X$ by the proof of Corollary 2.3. By Lemma 22 of \cite{HC}, $$\lim_{t\rightarrow 0^+} \partial(\nu)^n|_{t\nu}\mathcal{O}_{t\nu}=\lim_{t\rightarrow 0^+} \partial(\nu)^n|_{t\lambda}\mathcal{O}_{t\lambda}$$ on $G\cdot S_X$ if $\nu,\lambda\in C$. Clearly the support of $\lim_{t\rightarrow 0^+} \partial(\nu)^n|_{t\lambda}\mathcal{O}_{t\lambda}$ must be contained in $\mathcal{N}_{\lambda}$. Thus, the explicit formula for the limit on the left on $G\cdot S_X$ for open orbits $\mathcal{O}_X\subset \mathcal{N}_{\nu}$ in the proof of Corollary 2.3 then implies $\mathcal{O}_X\subset \mathcal{N}_{\lambda}$ whenever $\mathcal{O}_X\subset \mathcal{N}_{\nu}$ is open. Thus, we deduce $\mathcal{N}_{\nu}\subset \mathcal{N}_{\lambda}$. By symmetry we have equality.

\section{Rossmann's Character Formula}

\indent Let $V$ be a finite dimensional, real vector space, and let $\mu$ be a smooth, rapidly decreasing measure on $V$. Then the Fourier transform of $\mu$ is defined to be
$$\widehat{\mu}(l)=\int_V e^{i\langle l,X\rangle}d\mu(X).$$
Note $\widehat{\mu}$ is a Schwartz function on $V^*$. Given a tempered distribution $D$ on $V^*$, its Fourier transform $\widehat{D}$ is a tempered, generalized function on $V$ defined by 
$$\langle \widehat{D}, \mu\rangle:=\langle D,\widehat{\mu}\rangle.$$

\indent Let $G$ be a real, reductive algebraic group, and let $\pi$ be an irreducible, tempered representation of $G$ with character $\Theta_{\pi}$. Let $\theta_{\pi}$ be the Lie algebra analogue of the character of $\pi$. Rossmann associated to $\pi$ a finite union of regular, coadjoint orbits $\mathcal{O}_{\pi}\subset \mathfrak{g}^*$. He proved the following theorem \cite{R4}, \cite{R2}.

\begin{theorem} [Rossmann] As generalized functions, we have $$\widehat{\mathcal{O}_{\pi}}=\theta_{\pi}.$$
Here $\widehat{\mathcal{O}_{\pi}}$ denotes the Fourier transform of the canonical measure on $\mathcal{O}_{\pi}$. 
\end{theorem}

\indent If $\mathcal{O}_{\pi}=\mathcal{O}_{\nu}$ is a single orbit, then the leading term of the asymptotic expansion of $\widehat{\mathcal{O}_{\pi}}=\widehat{\mathcal{O}_{\nu}}$ is 
$$\sum_{\substack{\mathcal{O}_X\subset \mathcal{N}_{\nu}\\ \dim \mathcal{O}_X=\dim \mathcal{N}_{\nu}}} \text{vol}(\mathcal{O}_{\pi}\cap S_X)\widehat{\mathcal{O}_X}$$
by Corollary 2.3. More generally, suppose $\mathcal{O}_{\pi}=\cup \mathcal{O}_{\nu_i}$, and define $\operatorname{AS}(\pi)=\cup \mathcal{N}_{\nu_i}$. Then the leading term of $\widehat{\mathcal{O}_{\pi}}$ is a sum over the leading terms of the $\widehat{\mathcal{O}_{\nu_i}}$ of minimal degree. This is just the sum over $\text{vol}(\mathcal{O}_{\pi}\cap S_X)\widehat{\mathcal{O}_X}$ where $\mathcal{O}_X\subset \operatorname{AS}(\pi)$ varies over orbits of maximal dimension. Hence, we have shown 
$$\text{WF}(\pi)=\sum_{\substack{\mathcal{O}_X\subset \operatorname{AS}(\pi)\\ \dim \mathcal{O}_X=\dim \operatorname{AS}(\pi)}} \text{vol}(\mathcal{O}_{\pi}\cap S_X)\mathcal{O}_X.$$ 
\indent Thus, to prove Theorem 1.3, we need only show that $\mathcal{O}_X\subset \operatorname{AS}(\pi)$ is of maximal dimension iff $\mathcal{O}_{\pi}\cap S_X$ is precompact and non-empty. If $\mathcal{O}_X\subset \operatorname{AS}(\pi)$ is of maximal dimension, then $\mathcal{O}_X\cap \mathcal{N}_{\nu_i}$ is open for every $i$ and non-empty for at least one $i$. And by Corollary $2.2$, $\mathcal{O}_{\nu_i}\cap S_X$ is precompact for every $i$ and non-empty for at least one $i$. We conclude $\mathcal{O}_{\pi}\cap S_X$ is precompact and non-empty.\\
\indent Conversely, suppose $\mathcal{O}_{\pi}\cap S_X$ is precompact and non-empty. Then $\mathcal{O}_{\nu_i}\cap S_X$ is precompact for every $i$ and non-empty for some $i$. Corollary 2.2 tells us that $\mathcal{O}_X\subset \mathcal{N}_{\nu_i}$ is open for some $i$ and $\mathcal{O}_X\cap \mathcal{N}_{\nu_i}$ is open for every $i$. Hence, we have that $\mathcal{O}_X\subset \operatorname{AS}(\pi)$ is open. But, by Theorem $D$ of \cite{R3}, $\operatorname{AS}(\pi)$ is the closure of the union of the orbits in $\operatorname{AS}(\pi)$ of maximal dimension. Thus, $\mathcal{O}_X\subset \operatorname{AS}(\pi)$ open implies that $\mathcal{O}_X\subset \operatorname{AS}(\pi)$ is of maximal dimension.

\section{Proof of Theorem 1.1}

\indent In this section we prove Theorem 1.1. It is enough to prove the theorem when $\pi$ is an irreducible, tempered representation. Let $\mathcal{O}_X\subset \operatorname{AS}(\pi)$ be an open orbit, and identify $\mathfrak{g}\cong \mathfrak{g}^*$ via a $G$-equivariant isomorphism. Let $L\subset Z_G(X)$ be a Levi factor. Then there exists an $\mathfrak{sl}_2$-triple $\{X,H,Y\}$ such that $L=Z_G\{X,H,Y\}$. To prove the theorem, we must show
$$Z_G\{X,H,Y\}/Z(G)\ \text{is\ compact.}\ \ \ (*)$$\\
\indent Now, if $\mathcal{O}_{\pi}=\cup \mathcal{O}_{\nu_i}$, then $\operatorname{AS}(\pi)=\cup \mathcal{N}_{\nu_i}$. Any open orbit in $\operatorname{AS}(\pi)$ must be open in some $\mathcal{N}_{\nu_i}$. Hence, it is enough to prove $(*)$ whenever $\mathcal{O}_X$ is open in $\mathcal{N}_{\nu}$ for a regular element $\nu\in \mathfrak{g}^*$.\\ 
\indent Next, supplement A and supplement C of \cite{R1} imply that every regular orbit $\mathcal{O}_{\nu}$ can be written as a limit of regular semisimple orbits in the following sense. Let $\xi\in \mathfrak{g}^*\cong \mathfrak{g}$ be a semisimple element in the closure of $\mathcal{O}_{\nu}$, and let $\mathfrak{h}\subset \mathfrak{g}$ be a fundamental Cartan in $Z_{\mathfrak{g}}(\xi)$. Then there exists a connected component $C\subset (\mathfrak{h}^*)'$ such that 
$$\lim_{\substack{\lambda \in C\\ \lambda\rightarrow \xi}} \mathcal{O}_{\lambda}=\mathcal{O}_{\nu}.$$ Hence, $\mathcal{N}_{\nu}=\mathcal{N}_{\lambda}$ for some regular semisimple element $\lambda$, and it is enough to prove $(*)$ for $\mathcal{O}_X$ open in $\mathcal{N}_{\lambda}$ with $\lambda$ regular, semisimple.\\
\indent Fix $X$ and suppose $\mathcal{O}_X\subset \mathcal{N}_{\lambda}$ is open for some $\lambda$ regular semisimple. By Corollary 2.2, $\mathcal{O}_{\lambda}\cap S_X$ is compact and non-empty. Now, $L=Z_G\{X,H,Y\}$ acts on this space, and $L$ must have at least one closed orbit (for instance, one can take an orbit of minimal dimension). Without loss of generality, we make it $L\cdot \lambda\subset \mathcal{O}_{\lambda}\cap S_X$ . Choose a Cartan $\mathfrak{h}$ and a component of the regular set $C\subset (\mathfrak{h}^*)'$ such that $\lambda\in C$. If $\xi\in U_1=G\cdot C$, then $\mathcal{O}_X\subset \mathcal{N}_{\lambda}=\mathcal{N}_{\xi}$ is open by Corollary 2.4. It then follows from Corollary 2.2 that $\mathcal{O}_{\xi}\cap S_X$ is compact for all $\xi$ in the open set $U_1$. Define $U=U_1\cap S_X$, an open subset of $S_X$.\\ 
\indent Now, $L/Z_G(\lambda)\cong L\cdot \lambda\subset \mathcal{O}_{\lambda}\cap S_X$ is a closed subset of a compact set; hence, $L/Z_G(\lambda)$ is compact. Note that $Z_G(\lambda)\subset G$ is a Cartan since $\lambda$ is regular, semisimple. Thus, $Z_L(\lambda)\subset Z_G(\lambda)$ is abelian and consists of semisimple elements. Hence, the connected component of the identity $Z_L(\lambda)_0$ must be contained in a Cartan $B$ of $L$. We have seen that $L/Z_L(\lambda)$ is compact; hence, $L/Z_L(\lambda)_0$ is compact because $Z_L(\lambda)$ has finitely many components. This implies $L/B$ is compact and finally 
$$L/Z(L)$$ is compact since semisimple groups are compact iff they are compact modulo a Cartan.\\
\indent Because the fibers of the projection 
$$L/Z(G)\rightarrow L/Z(L)$$
are homeomorphic to $Z(L)/Z(G)$, to show that $L/Z(G)$ is compact, it is enough to show $Z(L)/Z(G)$ is compact.\\
The following lemma is the key step in proving that $Z(L)/Z(G)$ is compact.

\begin{lemma} Let $Z(\mathfrak{l})$ denote the center of $\mathfrak{l}=\text{Lie}(L)$, and let $Z(\mathfrak{g})$ denote the center of $\mathfrak{g}$. Then
$$\bigcap_{\xi\in U}(Z(\mathfrak{l})\cap Z_{\mathfrak{l}}(\xi))=Z(\mathfrak{g}).$$
\end{lemma}

\begin{proof} Clearly the right hand side is contained in the left hand side. To show the other direction, suppose $W\in (Z(\mathfrak{l})\cap Z_{\mathfrak{l}}(\xi))$ for all $\xi\in U$. We will show $W\in Z(\mathfrak{g})$. Since $W\in \mathfrak{l}$, we know $X,H,Y\in Z_{\mathfrak{g}}(W)$. Further, $Z_{\mathfrak{g}}(W)\cap Z_{\mathfrak{g}}(Y)\subset Z_{\mathfrak{g}}(Y)$ is a vector subspace containing $U-X$ since $W\in Z_{\mathfrak{l}}(\xi)$ for $\xi\in U$ and $X\in Z_{\mathfrak{g}}(W)$. Since $U-X\subset Z_{\mathfrak{g}}(Y)$ is an open subset, we must have 
$$Z_{\mathfrak{g}}(W)\supset Z_{\mathfrak{g}}(Y).$$
\noindent Now, view $\mathfrak{g}$ as a finite dimensional module for $\text{Span}_{\mathbb{R}}\{X,H,Y\}\cong \mathfrak{sl}_2\mathbb{R}$. Note $Z_{\mathfrak{g}}(W)\subset \mathfrak{g}$ is a subalgebra and a submodule for $\text{Span}_{\mathbb{R}}\{X,H,Y\}$ since $X,H,Y\in Z_{\mathfrak{g}}(W)$. But, the lowest weight vectors of each irreducible summand of $\mathfrak{g}$ are in $Z_{\mathfrak{g}}(W)$ since $Z_{\mathfrak{g}}(W)\supset Z_{\mathfrak{g}}(Y)$, and the lowest weight vectors of any finite dimensional $\mathfrak{sl}_2$ module generate the entire module. Thus, $Z_{\mathfrak{g}}(W)=\mathfrak{g}$ and $W\in Z(\mathfrak{g})$ as desired.
\end{proof}

Before we get back to showing that $Z(L)/Z(G)$ is compact, we need two general remarks. First, suppose $A$ is an abelian, real algebraic group, suppose $\phi:\ A\rightarrow \operatorname{Aut}(V)$ is a representation of $A$ on a real vector space $V$, and suppose $S\subset V$ is a compact $A$-stable subset of $V$. Then $A$ acts on $S$ with compact orbits. This can be proved as follows. After complexifying the representation, we may diagonalize the action of the image $\phi(A)$ since it is abelian and consists of semisimple elements. Now, $A$ must be isomorphic to a product of copies of $S^1$, $\mathbb{R}^{\times}$, and $\mathbb{C}^{\times}$. Using that every one dimensional character of these groups has either compact or unbounded image in $\mathbb{C}$, we deduce that every orbit of $A$ on $V\otimes \mathbb{C}$ is either compact or unbounded. In particular, $A$ must act on a compact $S\subset V$ with compact orbits.\\
\\
\indent Second, if $A$ is an abelian, real algebraic group and $A_1,A_2$ are cocompact, algebraic, closed subgroups, then $A_1\cap A_2$ is cocompact in $A$. This is because the fibers of the map
$$A/(A_1\cap A_2)\rightarrow A/A_1$$ are homeomorphic to $A_1/(A_1\cap A_2)\cong A_1A_2/A_2$, which is compact because it is a closed subset of $A/A_2$. More generally, if $A_1,\ldots,A_n$ is a finite collection of cocompact, algebraic, closed subgroups of a real, abelian algebraic group $A$, then $$A/\cap_{i=1}^n A_i$$ is compact.\\
\\
\indent Now, back to the proof that $Z(L)/Z(G)$ is compact. By the first remark, $Z(L)$ acts on $\mathcal{O}_{\xi}\cap S_X$ with compact orbits for every $\xi\in U$. In particular, $Z(L)/(Z(L)\cap Z_L(\xi))$ is compact for all $\xi\in U$. In the above lemma, we showed
$$\bigcap_{\xi\in U}(Z(\mathfrak{l})\cap Z_{\mathfrak{l}}(\xi))=Z(\mathfrak{g}).$$
However, one can clearly choose $\xi_1,\ldots,\xi_k\in U$ such that the identity still holds when taking the intersection over this finite set. Then, by the second remark,   
$$Z(L)/\bigcap_{i=1}^k(Z(L)\cap Z_L(\xi_i))$$ is compact. Since $\bigcap_{i=1}^k(Z(L)\cap Z_L(\xi_i))$ is a real algebraic group, it has a finite number of connected components and $$Z(L)/(\bigcap_{i=1}^k(Z(L)\cap Z_L(\xi_i)))_0$$ is also compact where $(\bigcap_{i=1}^k(Z(L)\cap Z_L(\xi_i)))_0$ denotes the identity component. But, since $\bigcap_{i=1}^k(Z(L)\cap Z_L(\xi_i))$ and $Z(G)$ share a Lie algebra, $Z(L)/Z(G)$ is a quotient of $$Z(L)/(\bigcap_{i=1}^k(Z(L)\cap Z_L(\xi_i)))_0.$$ Thus, $Z(L)/Z(G)$ is compact. This completes the proof of Theorem 1.1.

\section{Two Examples}
\noindent In this section, we give two examples. The author learned these examples from David Vogan.

\begin{example} Nilpotent orbits in $\text{U}(2,2)$ can be parametrized by signed Young diagrams of signature $(2,2)$ (see page 140 of \cite{CM}). A nilpotent orbit of $\text{U}(2,2)$ is real distinguished iff it corresponds to a signed Young diagram for which any two rows of equal length begin with the same sign. One can read this off of an explicit description of the Levi factor of the centralizer of a nilpotent element in the corresponding nilpotent orbit, which in turn can be easily computed from Remark 5.1.18 on page 74 of \cite{CM} and the discussion on pages 139-140 of \cite{CM}.\\
On the other hand, which nilpotent orbits of $\text{U}(2,2)$ occur in the wave front cycle of a discrete series can be read off from Example 2.6.7 on page 363 of \cite{Y}. In Yamamoto's table, the far left column is filled in by sequences of numbers and signs. The entry corresponds to a discrete series representation if the sequence contains only signs. One then observes from Yamamoto's table that the nilpotent orbits corresponding to discrete series for $\text{U}(2,2)$ are precisely the real distinguished nilpotent orbits of $\text{U}(2,2)$.\\
Of course, Yamamoto is really computing the associated variety of a discrete series representation. That the associated variety of the representation corresponds to the wave front cycle of the representation under the Kostant-Sekiguchi correspondence is a deep result of \cite{SV}.
\end{example}

\begin{example} Nilpotent orbits in $\text{Sp}(4,\mathbb{C})$ can be parametrized by partitions of four whose odd parts occur with even multiplicity (see page 70 of \cite{CM}). Of these orbits, the principal orbit (corresponding to the trivial partition $4$) and the subregular orbit (corresponding to the partition $2^2$) are the real distinguished ones as can be read off of page 88 of \cite{CM}.\\
However, the wave front cycle of any tempered representations of $\text{Sp}(4,\mathbb{C})$ contains only the principal orbit. Therefore, in this case, the set of nilpotent orbits corresponding to tempered representations is a proper subset of the set of real distinguished nilpotent orbits.
\end{example}

\section{Acknowledgments}
\noindent The author would like to thank his advisor, David Vogan, for suggesting the relationship between tempered representations and nilpotent orbits that is explored in this paper. The author would also like to thank David Vogan for many helpful comments and corrections. The author would like to thank Esther Galina for pointing out a mistake in a remark relating the results of this paper with the results of No\"{e}l. We have fixed this error in this draft.

\bibliographystyle{amsplain}

\end{document}